\documentclass[11pt, final]{article}
\usepackage{a4}
\usepackage{amsmath}%
\usepackage{amstext}%
\usepackage{amssymb}%
\usepackage{showkeys}%
\usepackage{epsfig}%
\usepackage{cite}
\usepackage{tikz}
\usepackage{subfig}
\usepackage{caption}
\usepackage{multirow}
\usepackage{booktabs}
\usepackage{floatrow}

\usepackage{listings}
\newtheorem{theorem}{Theorem}

\newtheorem{axiom}{Axiom}

\newtheorem{corollary}[theorem]{Corollary}

\newtheorem{definition}[axiom]{Definition}

\newtheorem{lemma}[theorem]{Lemma}

\newtheorem{rem}[theorem]{Remark}
\newenvironment{remark}{\rem\rm}{\endrem}

\newcounter{unnumber}

\newtheorem{prf}[unnumber]{Proof.}
\newenvironment{proof}{\prf\rm}{\hfill{$\blacksquare$}\endprf}
\newcommand{\R}{\mathbb{R}}%
\newcommand{\N}{\mathbb{N}}%
\newcommand{\ol}{\overline}%

\renewcommand{\>}{\right\rangle}

\DeclareMathOperator*\dom{dom}%
\DeclareMathOperator*\id{Id}%
\DeclareMathOperator*\prox{prox}%
\DeclareMathOperator*\argmin{argmin}

\DeclareMathOperator*\crit{crit}
\DeclareMathOperator*\dist{dist}
\DeclareMathOperator*\proj{proj}

\title{A forward-backward dynamical approach to the minimization of the sum of a nonsmooth convex with a smooth nonconvex function}

\author{Radu Ioan Bo\c{t} \thanks{University of Vienna, Faculty of Mathematics, Oskar-Morgenstern-Platz 1, A-1090 Vienna, Austria,
email: radu.bot@univie.ac.at.} \and
Ern\"{o} Robert Csetnek \thanks {University of Vienna, Faculty of Mathematics, Oskar-Morgenstern-Platz 1, A-1090 Vienna, Austria,
email: ernoe.robert.csetnek@univie.ac.at. Research supported by FWF (Austrian Science Fund), Lise Meitner Programme, project M 1682-N25.}}

\begin{document}
\maketitle

\noindent \textbf{Abstract.} We address the minimization of the sum of a proper, convex and lower 
semicontinuous with a (possibly nonconvex) smooth function from the perspective of  an implicit dynamical system of forward-backward type. The latter is formulated by means of the gradient of the smooth function and of the proximal point operator  of the nonsmooth one. The trajectory generated by the dynamical system is proved to asymptotically converge to a critical point of the objective, provided a regularization of the latter satisfies the  Kurdyka-\L{}ojasiewicz property. Convergence rates for the trajectory in terms of the \L{}ojasiewicz exponent of the regularized  objective function are also provided. 
\vspace{1ex}

\noindent \textbf{Key Words.} dynamical systems, continuous forward-backward method, nonsmooth optimization,  limiting subdifferential, Kurdyka-\L{}ojasiewicz property
\vspace{1ex}

\noindent \textbf{AMS subject classification.} 34G25, 47J25, 47H05, 90C26, 90C30, 65K10 

\section{Introduction}\label{sec-intr}
In this paper we approach the solving of the optimization problem 
\begin{equation}\label{intr-opt-pb}  \ \inf_{x\in\R^n}[f(x)+g(x)], \end{equation}
where $f:\R^n\to \R\cup\{+\infty\}$ is a proper, convex, lower semicontinuous function and 
$g:\R^n\to \R$ a (possibly nonconvex) Fr\'{e}chet differentiable function with $\beta$-Lipschitz continuous gradient for $\beta \geq 0$, i.e., $\|\nabla g(x) - \nabla g(y) \| \leq \|x-y\| \ \forall x,y \in \R^n$, by associating to it the 
implicit dynamical system 
\begin{equation}\label{intr-dyn-syst}\left\{
\begin{array}{ll}
\dot x(t) +  x(t)=\prox_{\eta f}\Big(x(t)-\eta\nabla g(x(t))\Big)\\
x(0)=x_0,
\end{array}\right.\end{equation}
where $\eta>0$, $x_0\in\R^n$ is chosen arbitrary and 
$\prox_{\eta f}:\R^n\rightarrow\R^n$, defined by $$\prox\nolimits_{\eta f}(y)=\argmin_{u\in\R^n}\left\{f(u)+\frac{1}{2\eta}\|u-y\|^2\right\}$$
is the \emph{proximal point operator} of $\eta f$. 

Due to the Lipschitz property of the proximal point operator, the existence and uniqueness of strong global solutions of the dynamical system \eqref{intr-dyn-syst} is ensured in the framework of the Cauchy-Lipschitz Theorem. 

The asymptotic analysis of the trajectories is carried out in the setting of functions satisfying the \emph{Kurdyka-\L{}ojasiewicz}  property (so-called \emph{KL functions}). To this large class belong functions with different analytic features. The techniques for proving the asymptotic convergence of the trajectories generated by \eqref{intr-dyn-syst} towards a critical point of the objective function of \eqref{intr-opt-pb}, expressed as a zero of the limiting (Mordukhovich) subdifferential, use three main ingredients (see \cite{b-sab-teb, att-b-sv2013, alv-att-bolte-red} for a similar approach in the discrete setting). Namely, we show a sufficient decrease property along the trajectories of a regularization of the objective function, the existence of a subgradient lower bound for the trajectories and,
finally, we obtain convergence by making use of the Kurdyka-\L{}ojasiewicz property of the objective function. The case when  the objective function is semi-algebraic follows as particular case of our analysis.
We close our investigations by establishing convergence rates for the trajectories expressed in terms of 
the \L{}ojasiewicz exponent of the regularized objective function. 

Let us mention that in the context of minimizing a (nonconvex) smooth function (wich corresponds to the case when in \eqref{intr-opt-pb} $f(x) = 0$ for all $x\in\R^n$) several first- and second-order gradient type dynamical systems have been investigated by  \L{}ojasiewicz \cite{lojasiewicz1963}, Simon \cite{simon}, Haraux and Jendoubi \cite{h-j}, Alvarez, Attouch, Bolte and  Redont \cite[Section 4]{alv-att-bolte-red}, Bolte, Daniilidis and Lewis \cite[Section 4]{b-d-l2006}, etc. In the aforementioned papers, the convergence of the trajectories  is obtained in the framework of KL functions. 

In what concerns implicit dynamical systems of the same type like \eqref{intr-dyn-syst}, let us first mention
that Bolte has studied in \cite{bolte-2003} the asymptotic convergence of the trajectories of  
\begin{equation}\label{intr-syst-bolte}\left\{
\begin{array}{ll}
\dot x(t)+x(t)=\proj_C\big(x(t)-\eta\nabla g(x(t))\big)\\
x(0)=x_0.
\end{array}\right.\end{equation}
where $g :\R^n \rightarrow \R^n$ is convex and differentiable with Lipschitz continuous gradient and $\proj_C$ denotes the projection operator on the nonempty, closed and convex set $C\subseteq\R^n$, towards a minimizer of $g$ over $C$. This corresponds to the case when in \eqref{intr-opt-pb} $f$ is the indicator function of $C$, namely,  $f(x)=0$ for $x\in C$ and $+\infty$ otherwise. We refer also to the work of Antipin \cite{antipin} for more statements and results concerning the dynamical system \eqref{intr-syst-bolte}. The approach of  \eqref{intr-opt-pb} by means of \eqref{intr-dyn-syst}, stated as a generalization of \eqref{intr-syst-bolte}, has been recently considered by Abbas and Attouch in \cite[Section 5.2]{abbas-att-arx14} in the full convex setting. Implicit dynamical systems related to both optimization problems and monotone inclusions have been  considered in the literature also by Attouch and Svaiter in \cite{att-sv2011}, Attouch, Abbas and Svaiter in \cite{abbas-att-sv} and  Attouch, Alvarez and Svaiter in \cite{att-alv-sv}. These investigations have been continued and extended in \cite{bb-ts-cont, b-c-dyn-KM, b-c-dyn-pen, b-c-dyn-sec-ord, b-c-conv-rate-cont}. 

Finally, we would like to emphasize that the time discretization of \eqref{intr-dyn-syst} leads to forward-backward iterative algorithm
\begin{equation}\label{discr}x_{k+1}=\prox\nolimits_{\eta f}(x_k-\eta\nabla g (x_k)) \ \forall k \geq 0,\end{equation}
where the starting point $x_0\in\R^n$ is arbitrarily chosen. This splitting method, which is in the convex setting well-understood (see for example 
\cite{bauschke-book}), has been investigated in several papers in the nonconvex setting for KL functions, too.  We refer the reader to 
\cite{attouch-bolte2009, att-b-red-soub2010, att-b-sv2013, b-sab-teb, b-c-inertial-nonc-ts, bcl, c-pesquet-r, f-g-peyp, h-l-s-t, ipiano}  for different techniques and ideas used for carrying out the convergence analysis 
of iterative schemes of same type like \eqref{discr} in the nonconvex setting. 

\section{Preliminaries}\label{sec2}

In this section we recall some notions and results which are needed throughout the paper. Let $\N= \{0,1,2,...\}$ be the set of nonnegative integers. For $n\geq 1$, the Euclidean scalar product and the induced norm on $\R^n$
are denoted by $\langle\cdot,\cdot\rangle$ and $\|\cdot\|$, respectively. Notice that all the finite-dimensional spaces considered in the  manuscript are endowed with the topology induced by the Euclidean norm. 

The {\it domain} of the function  $f:\R^n\rightarrow \R\cup\{+\infty\}$ is defined by $\dom f=\{x\in\R^n:f(x)<+\infty\}$. We say that $f$ is {\it proper} if $\dom f\neq\emptyset$.  
For the following generalized subdifferential notions and their basic properties we refer to \cite{boris-carte, rock-wets}. 
Let $f:\R^n\rightarrow \R\cup\{+\infty\}$ be a proper and lower semicontinuous function. If $x\in\dom f$, we consider the {\it Fr\'{e}chet (viscosity)  
subdifferential} of $f$ at $x$ as the set $$\hat{\partial}f(x)= \left \{v\in\R^n: \liminf_{y\rightarrow x}\frac{f(y)-f(x)-\<v,y-x\>}{\|y-x\|}\geq 0 \right \}.$$ For 
$x\notin\dom f$ we set $\hat{\partial}f(x):=\emptyset$. The {\it limiting (Mordukhovich) subdifferential} is defined at $x\in \dom f$ by 
$$\partial f(x)=\{v\in\R^n:\exists x_k\rightarrow x,f(x_k)\rightarrow f(x)\mbox{ and }\exists v_k\in\hat{\partial}f(x_k),v_k\rightarrow v \mbox{ as }k\rightarrow+\infty\},$$
while for $x \notin \dom f$, one takes $\partial f(x) :=\emptyset$. Therefore  $\hat\partial f(x)\subseteq\partial f(x)$ for each $x\in\R^n$.

Notice that in case $f$ is convex, these subdifferential notions coincide with the {\it convex subdifferential}, thus 
$\hat\partial f(x)=\partial f(x)=\{v\in\R^n:f(y)\geq f(x)+\<v,y-x\> \ \forall y\in \R^n\}$ for all $x\in\R^n$. 

Th graph of the limiting subdifferential fulfills the following closedness criterion: if $(x_k)_{k\in\N}$ and $(v_k)_{k\in\N}$ are sequences in $\R^n$ such that 
$v_k\in\partial f(x_k)$ for all $k\in\N$, $(x_k,v_k)\rightarrow (x,v)$ and $f(x_k)\rightarrow f(x)$ as $k\rightarrow+\infty$, then 
$v\in\partial f(x)$. 

The Fermat rule reads in this nonsmooth setting as follows: if $x\in\R^n$ is a local minimizer of $f$, then $0\in\partial f(x)$.  We denote by 
$$\crit(f)=\{x\in\R^n: 0\in\partial f(x)\}$$ the set of {\it (limiting)-critical points} of $f$. 

When $f$ is continuously differentiable around $x \in \R^n$ we have $\partial f(x)=\{\nabla f(x)\}$. We will make use of the following subdifferential sum rule:
if $f:\R^n\rightarrow\R\cup\{+\infty\}$ is proper and lower semicontinuous  and $h:\R^n\rightarrow \R$ is a continuously differentiable function, then $\partial (f+h)(x)=\partial f(x)+\nabla h(x)$ for all $x\in\R^m$. 

We turn now our attention to functions satisfying the {\it Kurdyka-\L{}ojasiewicz property}. This class of functions will play 
a crucial role in the asymptotic analysis of the dynamical system  \eqref{intr-dyn-syst}. For $\eta\in(0,+\infty]$, we denote by $\Theta_{\eta}$ the class of concave and continuous functions 
$\varphi:[0,\eta)\rightarrow [0,+\infty)$ such that $\varphi(0)=0$, $\varphi$ is continuously differentiable on $(0,\eta)$, continuous at $0$ and $\varphi'(s)>0$ for all 
$s\in(0, \eta)$. In the following definition (see \cite{att-b-red-soub2010, b-sab-teb}) we use also the {\it distance function} to a set, defined for $A\subseteq\R^n$ as $\dist(x,A)=\inf_{y\in A}\|x-y\|$  
for all $x\in\R^n$. 

\begin{definition}\label{KL-property} \rm({\it Kurdyka-\L{}ojasiewicz property}) Let $f:\R^n\rightarrow\R\cup\{+\infty\}$ be a proper and lower semicontinuous 
function. We say that $f$ satisfies the {\it Kurdyka-\L{}ojasiewicz (KL) property} at $\ol x\in \dom\partial f=\{x\in\R^n:\partial f(x)\neq\emptyset\}$, if there exist $\eta \in(0,+\infty]$, a neighborhood $U$ of $\ol x$ and a function $\varphi\in \Theta_{\eta}$ such that for all $x$ in the 
intersection 
$$U\cap \{x\in\R^n: f(\ol x)<f(x)<f(\ol x)+\eta\}$$ the following inequality holds 
$$\varphi'(f(x)-f(\ol x))\dist(0,\partial f(x))\geq 1.$$
If $f$ satisfies the KL property at each point in $\dom\partial f$, then $f$ is called {\it KL function}. 
\end{definition}

The origins of this notion go back to the pioneering work of \L{}ojasiewicz \cite{lojasiewicz1963}, where it is proved that for a real-analytic function 
$f:\R^n\rightarrow\R$ and a critical point $\ol x\in\R^n$ (that is $\nabla f(\ol x)=0$), there exists $\theta\in[1/2,1)$ such that the function 
$|f-f(\ol x)|^{\theta}\|\nabla f\|^{-1}$ is bounded around $\ol x$. This corresponds to the situation when $\varphi(s)=Cs^{1-\theta}$, where 
$C>0$. The result of 
\L{}ojasiewicz allows the interpretation of the KL property as a re-parametrization of the function values in order to avoid flatness around the 
critical points. Kurdyka \cite{kurdyka1998} extended this property to differentiable functions definable in o-minimal structures. 
Further extensions to the nonsmooth setting can be found in \cite{b-d-l2006, att-b-red-soub2010, b-d-l-s2007, b-d-l-m2010}. 

One of the remarkable properties of the KL functions is their ubiquity in applications (see \cite{b-sab-teb}). To the class of KL functions belong semi-algebraic, real sub-analytic, semiconvex, uniformly convex and 
convex functions satisfying a growth condition. We refer the reader to 
\cite{b-d-l2006, att-b-red-soub2010, b-d-l-m2010, b-sab-teb, b-d-l-s2007, att-b-sv2013, attouch-bolte2009} and the references therein for more on KL functions and illustrating examples. 

An important role in our convergence analysis will be played by the following uniformized KL property given in \cite[Lemma 6]{b-sab-teb}.

\begin{lemma}\label{unif-KL-property} Let $\Omega\subseteq \R^n$ be a compact set and let $f:\R^n\rightarrow\R\cup\{+\infty\}$ be a proper 
and lower semicontinuous function. Assume that $f$ is constant on $\Omega$ and that it satisfies the KL property at each point of $\Omega$.   
Then there exist $\varepsilon,\eta >0$ and $\varphi\in \Theta_{\eta}$ such that for all $\ol x\in\Omega$ and all $x$ in the intersection 
\begin{equation}\label{int} \{x\in\R^n: \dist(x,\Omega)<\varepsilon\}\cap \{x\in\R^n: f(\ol x)<f(x)<f(\ol x)+\eta\}\end{equation} 
the inequality  \begin{equation}\label{KL-ineq}\varphi'(f(x)-f(\ol x))\dist(0,\partial f(x))\geq 1.\end{equation}
holds.
\end{lemma}

In the following we recall the notion of locally absolutely continuous function and state two of its basic properties.

\begin{definition}\label{abs-cont} \rm (see, for instance, \cite{att-sv2011, abbas-att-sv}) 
A function $x : [0,+\infty) \rightarrow \R^n$ is said to be locally absolutely continuous, if it absolutely continuos on every interval $[0,T]$, where $T > 0$, which means
that one of the following equivalent properties holds:
\begin{enumerate}

\item[(i)] there exists an integrable function $y:[0,T]\rightarrow \R^n$ such that $$x(t)=x(0)+\int_0^t y(s)ds \ \ \forall t\in[0,T];$$

\item[(ii)] $x$ is continuous and its distributional derivative is Lebesgue integrable on $[0,T]$; 

\item[(iii)] for every $\varepsilon > 0$, there exists $\eta >0$ such that for any finite family of intervals $I_k=(a_k,b_k) \subseteq [0,T]$ 
we have the implication
$$\left(I_k\cap I_j=\emptyset \mbox{ and }\sum_k|b_k-a_k| < \eta\right)\Longrightarrow \sum_k\|x(b_k)-x(a_k)\| < \varepsilon.$$
\end{enumerate}
\end{definition}

\begin{remark}\label{rem-abs-cont}\rm\begin{enumerate} \item[(a)] It follows from the definition that an absolutely continuous function is differentiable almost 
everywhere, its derivative coincides with its distributional derivative almost everywhere and one can recover the function from its 
derivative $\dot x=y$ by the integration formula (i). 

\item[(b)] If $x:[0,T]\rightarrow {\cal H}$ is absolutely continuous for $T > 0$ and $B:\R^n\rightarrow \R^n$ is 
$L$-Lipschitz continuous for $L\geq 0$, then the function $z=B\circ x$ is absolutely continuous, too.
This can be easily seen by using the characterization of absolute continuity in
Definition \ref{abs-cont}(iii). Moreover, $z$ is differentiable almost everywhere on $[0,T]$ and the inequality 
$\|\dot z (t)\|\leq L\|\dot x(t)\|$ holds for almost every $t \in [0,T]$.  
\end{enumerate}
\end{remark}

The following two results, which can be interpreted as continuous versions of the quasi-Fej\'er monotonicity for sequences, 
will play an important role in the asymptotic analysis of the trajectories of the dynamical system investigated in this paper. 
For their proofs we refer the reader  to \cite[Lemma 5.1]{abbas-att-sv} and \cite[Lemma 5.2]{abbas-att-sv}, respectively.

\begin{lemma}\label{fejer-cont1} Suppose that $F:[0,+\infty)\rightarrow\R$ is locally absolutely continuous and bounded from below and that
there exists $G\in L^1([0,+\infty))$ such that for almost every $t \in [0,+\infty)$ $$\frac{d}{dt}F(t)\leq G(t).$$ 
Then there exists $\lim_{t\rightarrow \infty} F(t)\in\R$. 
\end{lemma}

\begin{lemma}\label{fejer-cont2}  If $1 \leq p < \infty$, $1 \leq r \leq \infty$, $F:[0,+\infty)\rightarrow[0,+\infty)$ is 
locally absolutely continuous, $F\in L^p([0,+\infty))$, $G:[0,+\infty)\rightarrow\R$, $G\in  L^r([0,+\infty))$ and 
for almost every $t \in [0,+\infty)$ $$\frac{d}{dt}F(t)\leq G(t),$$ then $\lim_{t\rightarrow +\infty} F(t)=0$. 
\end{lemma}

Further we recall a differentiability result involving the composition of convex functions with absolutely 
continuous trajectories which is due to Br\'{e}zis (\cite[Lemme 4, p. 73]{brezis}; see also \cite[Lemma 3.2]{att-cza-10}). 

\begin{lemma}\label{diff-brezis} Let $f:\R^n\rightarrow \R\cup\{+\infty\}$ be a proper, convex and lower semicontinuous function. 
Let $x\in L^2([0,T],\R^n)$ be absolutely continuous such that $\dot x\in L^2([0,T],\R^n)$ and $x(t)\in\dom f$ for almost every 
$t \in [0,T]$. Assume that there exists $\xi\in L^2([0,T],\R^n)$ such that $\xi(t)\in\partial f(x(t))$ for almost every $t \in [0,T]$. Then the function 
$t\mapsto f(x(t))$ is absolutely continuous and for every $t$ such that $x(t)\in\dom \partial f$ we have 
$$\frac{d}{dt}f(x(t))=\langle \dot x(t),h\rangle \ \forall h\in\partial f(x(t)).$$ 
\end{lemma}

We close this section by recalling the following characterization of the proximal point operator of a proper, convex and lower semincontinuous function
$f:\R^n\rightarrow\R\cup\{+\infty\}$. For every $\eta >0$ it holds (see for example \cite{bauschke-book})
\begin{equation}\label{ch-prox}p=\prox\nolimits_{\eta f}(x) \mbox{ if and only if }x\in p+\eta\partial f(p),\end{equation}
where $\partial f$ denotes the convex subdifferential of $f$.

\section{Asymptotic analysis}\label{sec3}

Before starting with the convergence analysis for the dynamical system \eqref{intr-dyn-syst}, we would like to point out that this can be written as

\begin{equation}\label{dyn-syst-M}\left\{
\begin{array}{ll}
\dot x(t) = (\prox\circ(\id-\eta\nabla g)-\id)\big(x(t)\big),\\
x(0)=x_0,
\end{array}\right.\end{equation}
where $\prox\circ(\id-\eta\nabla g)-\id $ is a $(2+\eta\beta)$-Lipschitz continuous operator. This follows from the fact that the proximal point operator of a proper, convex and lower semicontinuous function is 
nonexpansive, i.e., $1$-Lipschitz continuous (see for example \cite{bauschke-book}). According to the global version of the Cauchy-Lipschitz Theorem (see for instance \cite[Theorem 17.1.2(b)]{abm}), there exists
a unique global solution $x \in C^1([0,+\infty), \R^n)$  of the above dynamical system.

\subsection{Convergence of the trajectories}\label{subsec31}

\begin{lemma}\label{l-decr} Suppose that $f+g$ is bounded from below and $\eta>0$ fulfills the inequality
\begin{equation}\label{eta-beta}\eta\beta(3+\eta\beta)<1.\end{equation} For $x_0\in\R^n$, let  $x \in C^1([0,+\infty), \R^n)$ be the unique global solution of 
\eqref{intr-dyn-syst}. Then the following statements hold: 
\begin{enumerate}
                                            \item [(a)] $\dot x\in L^2([0,+\infty);\R^n)$ and $\lim_{t\rightarrow+\infty}\dot x(t)=0$;
                                            \item [(b)] $\exists\lim_{t\rightarrow+\infty}(f+g)\big(\dot x(t)+x(t)\big)\in\R$.
                                           \end{enumerate}
\end{lemma}

\begin{proof} Let us start by noticing that in the light of the the reformulation in \eqref{dyn-syst-M} and of Remark \ref{rem-abs-cont}(b), $\dot x$ is locally absolutely 
continuous, hence $\ddot x$ exists and for almost every $t \in [0,+\infty)$ one has 
\begin{equation}\label{ddot-exist}\|\ddot x(t)\|\leq(2+\eta\beta)\|\dot x(t)\|.\end{equation}

We fix an arbitrary $T>0$. Due to the continuity properties of the trajectory on $[0,T]$, \eqref{ddot-exist} and the Lipschitz continuity of $\nabla g$, one has
$$x, \dot x, \ddot x, \nabla g(x) \in L^2([0,T];\R^n).$$
Further, from the characterization  \eqref{ch-prox} of the proximal point operator we have 
\begin{equation}\label{from-def-prox}-\frac{1}{\eta}\dot x(t)-\nabla g(x(t))\in\partial f(\dot x(t)+x(t)) \ \forall t \in [0,+\infty).\end{equation}
Applying Lemma \ref{diff-brezis} we obtain that the function $t\mapsto f\big(\dot x(t)+x(t)\big)$ is absolutely continuous 
and $$\frac{d}{dt}f\big(\dot x(t)+x(t)\big)=\left\langle -\frac{1}{\eta}\dot x(t)-\nabla g(x(t)),\ddot x(t)+\dot x(t)\right\rangle$$
for almost every $t \in [0,T]$. Moreover, it holds 
$$\frac{d}{dt}g\big(\dot x(t)+x(t)\big)=\left\langle \nabla g\big(\dot x(t)+x(t)\big),\ddot x(t)+\dot x(t)\right\rangle$$
for almost every $t \in [0,T]$. Summing up the last two equalities, we obtain
\begin{align}\frac{d}{dt}(f+g)\big(\dot x(t)+x(t)\big) = & \left\langle -\frac{1}{\eta}\dot x(t)-\nabla g(x(t))+\nabla g\big(\dot x(t)+x(t)\big), 
\ddot x(t)+\dot x(t)\right\rangle\nonumber\\
 = & -\frac{1}{2\eta}\frac{d}{dt}\big(\|\dot x(t)\|^2\big)-\frac{1}{\eta}\|\dot x(t)\|^2\nonumber\\
 & +\left\langle \nabla g\big(\dot x(t)+x(t)\big)-\nabla g(x(t)), 
\ddot x(t)+\dot x(t)\right\rangle\nonumber\\
\label{lip-g}\leq & -\frac{1}{2\eta}\frac{d}{dt}\big(\|\dot x(t)\|^2\big)-\frac{1}{\eta}\|\dot x(t)\|^2
 + \beta\|\dot x(t)\|\cdot\|\ddot x(t)+\dot x(t)\|\\
 \label{ddot}\leq & -\frac{1}{2\eta}\frac{d}{dt}\big(\|\dot x(t)\|^2\big)-\frac{1}{\eta}\|\dot x(t)\|^2
 + \beta(3+\eta\beta)\|\dot x(t)\|^2\\
 = & -\frac{1}{2\eta}\frac{d}{dt}\big(\|\dot x(t)\|^2\big)-\left[\frac{1}{\eta}-\beta(3+\eta\beta)\right]\|\dot x(t)\|^2\nonumber
\end{align}
for almost every $t \in [0,T]$, where in \eqref{lip-g} we used the Lipschitz continuity of $\nabla g$ and in \eqref{ddot} the inequality \eqref{ddot-exist}. Altogether, we conclude that for almost every $t\in [0,T]$ we have 
\begin{equation}\label{decr-f} \frac{d}{dt}\left[(f+g)\big(\dot x(t)+x(t)\big)+\frac{1}{2\eta}\|\dot x(t)\|^2\right]+
\left[\frac{1}{\eta}-\beta(3+\eta\beta)\right]\|\dot x(t)\|^2\leq 0
\end{equation}
and by integration we get
\begin{align}\label{integ}
& (f+g)\big(\dot x(T)+x(T)\big)+\frac{1}{2\eta}\|\dot x(T)\|^2  + \left[\frac{1}{\eta}-\beta(3+\eta\beta)\right] \int_{0}^T \|\dot x(t)\|^2dt \leq \nonumber \\
& (f+g)\big(\dot x(0)+x(0)\big)+\frac{1}{2\eta}\|\dot x(0)\|^2. 
\end{align}
By using  \eqref{eta-beta} and the fact that $f+g$ is bounded from below and by taking into account that $T > 0$ has been arbitrarily chosen, we obtain
\begin{equation}\label{dot-l2}\dot x\in L^2([0,+\infty);\R^n).      \end{equation}
Due to \eqref{ddot-exist}, this further implies
\begin{equation}\label{ddot-l2}
\ddot x\in L^2([0,+\infty);\R^n).      
\end{equation}
Furthermore, for almost every $t \in [0,+\infty)$ we have 
$$\frac{d}{dt}\big(\|\dot x(t)\|^2\big)=2\langle \dot x(t),\ddot x(t)\rangle\leq \|\dot x(t)\|^2+\|\ddot x(t)\|^2.$$
By applying Lemma \ref{fejer-cont2}, it follows that $\lim_{t\rightarrow+\infty}\dot x(t)=0$ and the proof of (a) is complete. 
From \eqref{decr-f}, \eqref{eta-beta} and by using that $T >0$ has been arbitrarily chosen, we get 
$$\frac{d}{dt}\left[(f+g)\big(\dot x(t)+x(t)\big)+\frac{1}{2\eta}\|\dot x(t)\|^2\right] \leq 0$$
for almost every $t \in [0,+\infty)$. From Lemma \ref{fejer-cont1} it follows that
$$\lim_{t\rightarrow+\infty}\left[(f+g)\big(\dot x(t)+x(t)\big)+\frac{1}{2\eta}\|\dot x(t)\|^2\right]$$
exists and it is a real number, hence from $\lim_{t\rightarrow+\infty}\dot x(t)=0$ the conclusion follows. 
\end{proof}

We defined the limit set of $x$  as 
$$\omega (x)=\{\ol x\in\R^n:\exists t_k\rightarrow+\infty \mbox{ such that }x(t_k)\rightarrow\ol x \mbox{ as }k\rightarrow+\infty\}.$$

\begin{lemma}\label{l-lim-crit-f} Suppose that $f+g$ is bounded from below and $\eta>0$ fulfills the inequality \eqref{eta-beta}. 
For $x_0\in\R^n$, let  $x \in C^1([0,+\infty), \R^n)$ be the unique global solution of 
\eqref{intr-dyn-syst}. Then $$\omega(x)\subseteq \crit (f+g).$$ 
\end{lemma}

\begin{proof} Let $\ol x\in\omega (x)$ and $t_k\rightarrow+\infty \mbox{ be such that }x(t_k)\rightarrow\ol x 
\mbox{ as }k\rightarrow+\infty.$ From \eqref{from-def-prox} we have 
\begin{align}-\frac{1}{\eta}\dot x(t_k)-\nabla g(x(t_k))+\nabla g\big(\dot x(t_k)+x(t_k)\big)\in & \ \partial f\big(\dot x(t_k)+x(t_k)\big)+
\nabla g\big(\dot x(t_k)+x(t_k)\big)\nonumber\\\label{incl-tk} = & \ \partial (f+g)\big(\dot x(t_k)+x(t_k)\big) \ \forall k \in \N.
\end{align}

Lemma \ref{l-decr}(a) and the Lipschitz continuity of $\nabla g$ ensure that 
\begin{equation}\label{bor1} -\frac{1}{\eta}\dot x(t_k)-\nabla g(x(t_k))+\nabla g\big(\dot x(t_k)+x(t_k)\big)\rightarrow 0 \mbox{ as }k\rightarrow+\infty 
\end{equation}
and 
\begin{equation}\label{bor2} \dot x(t_k)+x(t_k)\rightarrow \ol x \mbox{ as }k\rightarrow+\infty. 
\end{equation}

We claim that \begin{equation}\label{bor3} \lim_{k\rightarrow+\infty}(f+g)\big(\dot x(t_k)+x(t_k)\big)=(f+g)(\ol x).\end{equation}
Due to the lower semicontinuity of $f$ it holds
\begin{equation}\label{from-f-lsc}\liminf_{k\rightarrow+\infty}f\big(\dot x(t_k)+x(t_k)\big)\geq f(\ol x).\end{equation}

Further, since \begin{align*} \dot x(t_k)+x(t_k)= & \argmin_{u\in\R^n}\left[f(u)+\frac{1}{2\eta}\left\|u-\big(x(t_k)-\eta\nabla g(x(t_k))\big)\right\|^2\right]\\
                                      = & \argmin_{u\in\R^n}\left[f(u)+\frac{1}{2\eta}\|u-x(t_k)\|^2+\langle u-x(t_k),\nabla g(x(t_k))\rangle\right] \end{align*}
we have the inequality 
\begin{align*}
 & \ f\big(\dot x(t_k)+x(t_k)\big)+\frac{1}{2\eta}\|\dot x(t_k)\|^2+\langle \dot x(t_k),\nabla g(x(t_k))\rangle\\
\leq & \ f(\ol x)+\frac{1}{2\eta}\|\ol x-x(t_k)\|^2+\langle \ol x-x(t_k), \nabla g(x(t_k))\rangle \ \forall k \in \N.
\end{align*}

Taking the limit as $k\rightarrow+\infty$ we derive by using again Lemma \ref{l-decr}(a) that 
\begin{equation*}\limsup_{k\rightarrow+\infty}f\big(\dot x(t_k)+x(t_k)\big)\leq f(\ol x),\end{equation*}
which combined with \eqref{from-f-lsc} implies 
\begin{equation*}\lim_{k\rightarrow+\infty}f\big(\dot x(t_k)+x(t_k)\big)= f(\ol x).\end{equation*}
By using \eqref{bor2}  and the continuity of $g$ we conclude that \eqref{bor3} is true. 

Altogether, from \eqref{incl-tk}, \eqref{bor1}, \eqref{bor2}, \eqref{bor3} and the closedness criteria of the limiting subdifferential we 
obtain $0\in\partial (f+g)(\ol x)$ and the proof is complete.  
\end{proof}

\begin{lemma}\label{l-h123} Suppose that $f+g$ is bounded from below and $\eta>0$ fulfills the inequality \eqref{eta-beta}. 
For $x_0\in\R^n$, let  $x \in C^1([0,+\infty), \R^n)$ be the unique global solution of 
\eqref{intr-dyn-syst} and consider the function
$$H:\R^n\times\R^n\to\R\cup\{+\infty\},\, H(u,v)=(f+g)(u)+\frac{1}{2\eta}\|u-v\|^2.$$
Then the following statements are true:
\begin{itemize}
\item[($H_1$)] for almost every $t\in [0,+\infty)$ it holds 
$$\frac{d}{dt}H\big(\dot x(t)+x(t),x(t)\big)\leq -\left[\frac{1}{\eta}-(3+\eta\beta)\beta\right]\|\dot x(t)\|^2\leq 0$$  and 
$$\exists\lim_{t\rightarrow +\infty}H\big(\dot x(t)+x(t),x(t)\big)\in\R;$$
\item[($H_2$)] for almost every $t\in [0,+\infty)$ it holds  $$z(t):=\left(-\nabla g(x(t))+\nabla g\big(\dot x(t)+x(t)\big),-\frac{1}{\eta}\dot x(t)\right)\in\partial H\big(\dot x(t)+x(t),x(t)\big)$$ 
and $$\|z(t)\|\leq \left(\beta+\frac{1}{\eta}\right)\|\dot x(t)\|;$$ 
\item[($H_3$)] for $\ol x\in\omega (x)$ and $t_k\rightarrow+\infty$ such that $x(t_k)\rightarrow\ol x$ as $k\rightarrow+\infty$, we have  
$H\big(\dot x(t_k)+x(t_k),x(t_k)\big)\rightarrow H(\ol x,\ol x)$ as $k\rightarrow+\infty$.
\end{itemize}
\end{lemma}

\begin{proof} (H1) follows from Lemma \ref{l-decr}. The first statement in (H2) is a consequence of \eqref{from-def-prox} and the relation
\begin{equation}\label{H-subdiff}\partial H(u,v)=\big(\partial (f+g)(u)+\eta^{-1}(u-v)\big)\times \{\eta^{-1}(v-u)\} \ \forall (u,v)\in\R^n\times\R^n,\end{equation}
while the second one is a consequence of the Lipschitz continuity of $\nabla g$. Finally, (H3) has been shown as intermediate step in the proof of Lemma \ref{l-lim-crit-f}. 
\end{proof}

\begin{lemma}\label{l} Suppose that $f+g$ is bounded from below and $\eta>0$ fulfills the inequality \eqref{eta-beta}. 
For $x_0\in\R^n$, let  $x \in C^1([0,+\infty), \R^n)$ be the unique global solution of 
\eqref{intr-dyn-syst} and consider the function
$$H:\R^n\times\R^n\to\R\cup\{+\infty\},\, H(u,v)=(f+g)(u)+\frac{1}{2\eta}\|u-v\|^2.$$
Suppose that  $x$ is bounded. Then the following statements are true:
\begin{itemize}
\item[(a)] $\omega(\dot x+x,x)\subseteq \crit(H)=\{(u,u)\in\R^n\times\R^n:u\in \crit(f+g)\}$; 
\item[(b)] $\lim_{t\to+\infty}\dist\Big(\big(\dot x(t)+x(t),x(t)\big),\omega\big(\dot x + x,x\big)\Big)=0$;
\item[(c)] $\omega\big(\dot x+x,x\big)$ is nonempty, compact and connected;
\item[(d)] $H$ is finite and constant on $\omega\big(\dot x+x,x\big).$
\end{itemize}
\end{lemma}

\begin{proof} (a), (b) and (d) are direct consequences Lemma \ref{l-decr}, Lemma \ref{l-lim-crit-f} and Lemma \ref{l-h123}.

Finally, (c) is a classical result from \cite{haraux}. We also refer the reader  to the proof of Theorem 4.1 in \cite{alv-att-bolte-red}, where it is shown that the properties of $\omega(x)$ of being nonempty, compact and connected are generic for bounded trajectories fulfilling  $\lim_{t\rightarrow+\infty}{\dot x(t)}=0$). 
\end{proof}

\begin{remark}\label{cond-x-bound} 
Suppose that $\eta>0$ fulfills the inequality \eqref{eta-beta} and $f+g$ is cocoercive, that is $$\lim_{\|u\|\rightarrow+\infty}(f+g)(u)=+\infty.$$ 
For $x_0\in\R^n$, let  $x \in C^1([0,+\infty), \R^n)$ be the unique global solution of 
\eqref{intr-dyn-syst}. Then $f+g$ is bounded from below and  
$x$ is bounded.  

Indeed, since $f+g$ is a proper, lower semicontinuous and coercive function, it follows that 
$\inf_{u\in\R^n}[f(u)+g(u)]$ is finite and the infimum is attained. Hence $f+g$ is bounded from below. On the other hand, from \eqref{integ} it follows
\begin{align*}(f+g)\big(\dot x(T)+x(T)\big)\leq & \ (f+g)\big(\dot x(T)+x(T)\big)+\frac{1}{2\eta}\|\dot x(T)\|^2\\
  \leq & \ (f+g)\big(\dot x(0)+x_0)\big)+\frac{1}{2\eta}\|\dot x(0)\|^2 \ \forall T \geq 0.
\end{align*}
Since the lower level sets of $f+g$ are bounded, the above inequality yields the boundedness of $\dot x+x$, which 
combined with $\lim_{t\rightarrow+\infty}\dot x(t)=0$ delivers the boundedness of $x$. 
\end{remark}

We come now the main result of the paper. 

\begin{theorem}\label{conv-kl} Suppose that $f+g$ is bounded from below and $\eta>0$ fulfills the inequality \eqref{eta-beta}. For $x_0\in\R^n$, let  $x \in C^1([0,+\infty), \R^n)$ be the unique global solution of \eqref{intr-dyn-syst} and consider the function
$$H:\R^n\times\R^n\to\R\cup\{+\infty\},\, H(u,v)=(f+g)(u)+\frac{1}{2\eta}\|u-v\|^2.$$
Suppose that  $x$ is bounded and $H$ is a KL function. Then the following statements are true:
\begin{itemize}\item[(a)] $\dot x\in L^1([0,+\infty);\R^n)$;
\item[(b)] there exists $\ol x\in\crit(f+g)$ such that $\lim_{t\rightarrow+\infty}x(t)=\ol x$.
\end{itemize}
\end{theorem}

\begin{proof} According to Lemma \ref{l}, we can choose an element $\ol x\in\crit (f+g)$ such that 
$(\ol x,\ol x)\in \omega (\dot x+x,x)$. According to Lemma \ref{l-h123}, it follows that
$$\lim_{t\rightarrow+\infty}H\big(\dot x(t)+x(t),x(t)\big)=H(\ol x,\ol x).$$

We treat the following two cases separately. 

I. There exists $\ol t\geq 0$ such that $$H\big(\dot x(\ol t)+x(\ol t),x(\ol t)\big)=H(\ol x,\ol x).$$ Since from 
Lemma \ref{l-h123}(H1) we have $$\frac{d}{dt}H\big(\dot x(t)+x(t),x(t)\big) \leq 0 \ \forall t \in [0,+\infty),$$  we obtain  for every $t\geq \ol t$ that
$$H\big(\dot x(t)+x(t),x(t)\big)\leq H\big(\dot x(\ol t)+x(\ol t),x(\ol t)\big)=H(\ol x,\ol x).$$ Thus $H\big(\dot x(t)+x(t),x(t)\big)=H(\ol x,\ol x)$ for every $t\geq \ol t$. This yields by Lemma \ref{l-h123}(H1) that 
$\dot x(t)=0$ for almost every $t \in [\ol t, +\infty)$, hence $x$ is constant on $[\ol t,+\infty)$ and the conclusion follows. 

II. For every $t\geq 0$ it holds $H\big(\dot x(t)+x(t),x(t)\big)>H(\ol x,\ol x).$ Take $\Omega=\omega(\dot x+x,x)$. 

In virtue of Lemma \ref{l}(c) and (d) and since $H$ is a KL function, by Lemma \ref{unif-KL-property}, there exist positive numbers $\epsilon$ and $\eta$ and 
a concave function $\varphi\in\Theta_{\eta}$ such that for all
\begin{align}\label{int-H} 
(x,y)\in & \{(u,v)\in\R^n\times\R^n: \dist((u,v),\Omega)<\epsilon\} \nonumber \\ 
 & \cap\{(u,v)\in\R^n\times\R^n:H(\ol x,\ol x)<H(u,v)<H(\ol x,\ol x)+\eta\}\end{align}
one has
\begin{equation}\label{ineq-H}\varphi'(H(x,y)-H(\ol x,\ol x))\dist((0,0),\partial H(x,y))\ge 1.\end{equation}

Let $t_1\geq 0$ be such that $H\big(\dot x(t)+x(t),x(t)\big)<H(\ol x,\ol x)+\delta$ for all $t\geq t_1$. Since 
$\lim_{t\to+\infty}\dist\Big(\big(\dot x(t)+x(t),x(t)\big),\Omega\Big)=0$, there exists $t_2\geq 0$ such that 
$\dist\Big(\big(\dot x(t)+x(t),x(t)\big),\Omega\Big)<\epsilon$ for all $t\geq t_2$. Hence for all $t\geq T:=\max\{t_1,t_2\}$, 
$\big(\dot x(t)+x(t),x(t)\big)$ belongs to the intersection in \eqref{int-H}. Thus, according to \eqref{ineq-H}, for every $t\geq T$ we have
\begin{equation}\label{ineq-Ht1}\varphi'\Big(H\big(\dot x(t)+x(t),x(t)\big)-H(\ol x,\ol x)\Big)
\dist\Big((0,0),\partial H\big(\dot x(t)+x(t),x(t)\big)\Big)\ge 1.\end{equation}
By applying Lemma \ref{l-h123}(H2) we obtain for almost every $t \in [T, +\infty)$
\begin{equation}\label{ineq-Ht2}(\beta+\eta^{-1})\|\dot x(t)\|\varphi'\Big(H\big(\dot x(t)+x(t),x(t)\big)-H(\ol x,\ol x)\Big)
\ge 1.\end{equation}
From here, by using Lemma \ref{l-h123}(H1) and that $\varphi'>0$ and 
\begin{align*}
& \frac{d}{dt}\varphi\Big(H\big(\dot x(t)+x(t),x(t)\big)-H(\ol x,\ol x)\Big)=\\
& \varphi'\Big(H\big(\dot x(t)+x(t),x(t)\big)-H(\ol x,\ol x)\Big)\frac{d}{dt}H\big(\dot x(t)+x(t),x(t)\big),
\end{align*}
we deduce that for almost every $t \in [T, +\infty)$ it holds
\begin{equation}\label{ineq-pt-conv-r} \frac{d}{dt}\varphi\Big(H\big(\dot x(t)+x(t),x(t)\big)-H(\ol x,\ol x)\Big)\leq 
-\left(\beta+\eta^{-1}\right)^{-1}\left[\frac{1}{\eta}-(3+\eta\beta)\beta\right]\|\dot x(t)\|.\end{equation}
Since $\varphi$ is bounded from below, by taking into account \eqref{eta-beta}, it follows  $\dot x\in L^1([0,+\infty);\R^n)$. From here we obtain that $\lim_{t\rightarrow+\infty}x(t)$ exists and this closes the proof.
\end{proof}

Since the class of semi-algebraic functions is closed under addition (see for example \cite{b-sab-teb}) and 
$(u,v) \mapsto c\|u-v\|^2$ is semi-algebraic for $c>0$, we can stat the following direct consequence of the previous theorem.  

\begin{corollary}\label{conv-semi-alg}Suppose that $f+g$ is bounded from below and $\eta>0$ fulfills the inequality \eqref{eta-beta}. For $x_0\in\R^n$, let  $x \in C^1([0,+\infty), \R^n)$ be the unique global solution of \eqref{intr-dyn-syst}. Suppose that  $x$ is bounded and $f+g$ is semi-algebraic. Then the following statements are true:
\begin{itemize}\item[(a)] $\dot x\in L^1([0,+\infty);\R^n)$;
\item[(b)] there exists $\ol x\in\crit(f+g)$ such that $\lim_{t\rightarrow+\infty}x(t)=\ol x$.
\end{itemize}
\end{corollary}

\subsection{Convergence rates}\label{subsec32}

In this subsection we investigate the convergence rates of the trajectories generated by the dynamical system \eqref{intr-dyn-syst}. When solving optimization problems involving KL functions, convergence rates have been proved to depend on the so-called  \L{}ojasiewicz exponent  (see \cite{lojasiewicz1963, b-d-l2006, attouch-bolte2009, f-g-peyp}). The main result of this subsection refer to the KL functions which satisfy Definition \ref{KL-property}  for $\varphi(s)=Cs^{1-\theta}$, where $C>0$ and $\theta\in(0,1)$. We recall the following definition considered in \cite{attouch-bolte2009}. 

\begin{definition}\label{kl-phi} \rm Let $f:\R^n\rightarrow\R\cup\{+\infty\}$ be a proper and lower semicontinuous function. 
The function $f$ is said to have the \L{}ojasiewicz property, if for every $\ol x\in\crit f$ there exist $C,\varepsilon >0$ and 
$\theta\in(0,1)$ such that 
\begin{equation}\label{kl-phi-ineq}|f(x)-f(\ol x)|^{\theta}\leq C\|x^*\| \ \mbox{for every} \ x \ \mbox{fulfilling} \ \|x-\ol x\|<\varepsilon \mbox{ and every} \ x^*\in\partial f(x).\end{equation}
\end{definition}

According to \cite[Lemma 2.1 and Remark 3.2(b)]{att-b-red-soub2010}, the KL property is automatically 
satisfied at any noncritical point, fact which motivates the restriction to critical points in the above definition. The real number $\theta$ in the above definition is called \emph{\L{}ojasiewicz exponent} of the function $f$ at the critical point  $\ol x$. 

\begin{theorem}\label{conv-r} Suppose that $f+g$ is bounded from below and $\eta>0$ fulfills the inequality \eqref{eta-beta}. For $x_0\in\R^n$, let  $x \in C^1([0,+\infty), \R^n)$ be the unique global solution of \eqref{intr-dyn-syst} and consider the function
$$H:\R^n\times\R^n\to\R\cup\{+\infty\},\, H(u,v)=(f+g)(u)+\frac{1}{2\eta}\|u-v\|^2.$$
Suppose that  $x$ is bounded and $H$ satisfies Definition \ref{KL-property}  for $\varphi(s)=Cs^{1-\theta}$, where $C>0$ and $\theta\in(0,1)$. Then there exists $\ol x\in\crit (f+g)$ such that 
$\lim_{t\rightarrow+\infty}x(t)=\ol x.$ Let $\theta$ be the \L{}ojasiewicz exponent of $H$ at $(\ol x,\ol x)\in\crit H$, according to the Definition \ref{kl-phi}. Then there exist 
$a,b,c,d>0$ and $t_0\geq 0$ such that for every $t\geq t_0$ the following statements are true: 
\begin{itemize}\item[(a)] if $\theta\in (0,\frac{1}{2})$, then $x$ converges in finite time;
\item[(b)] if $\theta=\frac{1}{2}$, then $\|x(t)-\ol x\|\leq a\exp(-bt)$;
\item[(c)] if $\theta\in (\frac{1}{2},1)$, then $\|x(t)-\ol x\|\leq (ct+d)^{-\left(\frac{1-\theta}{2\theta-1}\right)}$.
\end{itemize}
\end{theorem}

\begin{proof} We define for every $t \geq 0$ (see also \cite{b-d-l2006})  
$$\sigma(t)=\int_{t}^{+\infty}\|\dot x(s)\|ds \ \mbox{ for all }t\geq 0.$$
It is immediate that 
\begin{equation}\label{x-sigma}\|x(t)-\ol x\|\leq \sigma(t) \ \forall t\geq 0.\end{equation}

Indeed, this follows by noticing that for $T \geq t$
\begin{align*}\|x(t)-\ol x\|= \ & \|x(T)-\ol x-\int_t^T\dot x(s)ds\|\\
                          \leq  & \ \|x(T)-\ol x\|+\int_{t}^T\|\dot x(s)\|ds,
                          \end{align*}
and by letting afterwards $T\rightarrow +\infty$.  

We assume that for every $t\geq 0$ we have $H\big(\dot x(t)+x(t),x(t)\big)>H(\ol x,\ol x).$ As seen in the proof of  Theorem \ref{conv-kl}, in the other case the conclusion follows automatically. Furthermore, by invoking again the proof of above-named result, there exist $t_0\geq 0$ and $M>0$ such that for every $t\geq t_0$ (see \eqref{ineq-pt-conv-r})
\begin{equation*} M \|\dot x(t)\|+ \frac{d}{dt}\left[\Big(H\big(\dot x(t)+x(t),x(t)\big)-H(\ol x,\ol x)\Big)\right]^{1-\theta}\leq 0\end{equation*}
and
\begin{equation*} \| \big(\dot x(t)+x(t),x(t)\big)-(\ol x,\ol x)\|<\varepsilon.\end{equation*}

We derive by integration (for $T\geq t\geq t_0$) $$M\int_t^{T}\|\dot x(s)\|ds+\left[\Big(H\big(\dot x(T)+x(T),x(T)\big)-H(\ol x,\ol x)\Big)\right]^{1-\theta} $$$$
\leq\left[\Big(H\big(\dot x(t)+x(t),x(t)\big)-H(\ol x,\ol x)\Big)\right]^{1-\theta},$$
hence 
\begin{equation}\label{ineq1}M\sigma (t)\leq \left[\Big(H\big(\dot x(t)+x(t),x(t)\big)-H(\ol x,\ol x)\Big)\right]^{1-\theta} \ \forall t\geq t_0.\end{equation}

Since  $\theta$ is the \L{}ojasiewicz exponent of $H$ at $(\ol x,\ol x)$, we have 
$$|H\big(\dot x(t)+x(t),x(t)\big)-H(\ol x,\ol x)|^{\theta}\leq C\|x^*\| \ \forall x^*\in \partial H\big(\dot x(t)+x(t),x(t)\big)$$
for every $t\geq t_0$. 
According to Lemma \ref{l-h123}(H2), we can find a constant $N>0$ such that for almost every $t \in [t_0, +\infty)$ there exists $x^*(t)\in \partial H\big(\dot x(t)+x(t),x(t)\big) $ fulfilling
$$\|x^*(t)\|\leq N\|\dot x(t)\|.$$
From the above two inequalities we derive for almost every $t \in [t_0, +\infty)$
$$|H\big(\dot x(t)+x(t),x(t)\big)-H(\ol x,\ol x)|^{\theta}\leq C\cdot N\|\dot x(t)\|,$$ 
which combined with \eqref{ineq1} yields 
\begin{equation}\label{ineq2}M\sigma (t)\leq (C\cdot N\|\dot x(t)\|)^{\frac{1-\theta}{\theta}}.\end{equation}

Since \begin{equation}\label{dsigma}\dot \sigma (t)=-\|\dot  x(t)\|\end{equation} we conclude that there exists $\alpha>0$ such that for almost every $t \in [t_0, +\infty)$
\begin{equation}\label{sigma} \dot\sigma (t)\leq -\alpha\big(\sigma(t)\big)^{\frac{\theta}{1-\theta}}. 
\end{equation}

If $\theta=\frac{1}{2}$, then $$\dot\sigma (t)\leq -\alpha\sigma(t)$$ for almost every $t \in [t_0, +\infty)$. By multiplying with $\exp(\alpha t)$ and integrating afterwards from $t_0$ to $t$, it follows that there exist $a,b>0$ such that 
$$\sigma (t)\leq a\exp(-bt) \ \forall t\geq t_0$$ and the conclusion of (b) is immediate from \eqref{x-sigma}. 

Assume that $0<\theta<\frac{1}{2}$. We obtain from \eqref{sigma} 
$$\frac{d}{dt}\left(\sigma^{\frac{1-2\theta}{1-\theta}}\right)\leq-\alpha \frac{1-2\theta}{1-\theta}$$
for almost every $t \in [t_0, +\infty)$.

By integration we get $$\big(\sigma(t)\big)^{\frac{1-2\theta}{1-\theta}}\leq -\ol \alpha t+\ol \beta \ \forall t\geq t_0,$$ 
where $\ol \alpha>0$. Thus there exists $T\geq 0$ such that $$\sigma (T)\leq 0 \  \forall t\geq T,$$
which implies that $x$ is constant on $[T,+\infty)$. 

Finally, suppose that $\frac{1}{2}<\theta<1$. We obtain from \eqref{sigma} 
$$\frac{d}{dt}\left(\sigma^{\frac{1-2\theta}{1-\theta}}\right)\geq\alpha \frac{2\theta-1}{1-\theta}.$$
By integration one derives $$\sigma(t)\leq (ct+d)^{-\left(\frac{1-\theta}{2\theta-1}\right)} \ \forall t\geq t_0,$$ where $c,d>0$, and (c) follows from \eqref{x-sigma}.
\end{proof}

\end{document}